\documentclass[11pt,a4paper]{amsart}
\usepackage[utf8]{inputenc}
\usepackage{amsmath}
\usepackage{amsfonts}
\usepackage{amssymb}
\usepackage{amsthm}
\usepackage{stmaryrd}
\usepackage{graphicx}
\usepackage{tikz}
\usepackage{tikz-cd}
\usepackage{mathtools}
\usepackage{comment}
\usepackage{import}

\usepackage[
bookmarks         = true,
bookmarksnumbered = true,
pdfstartview      = FitH,
pdfpagelayout     = OneColumn,
pdfpagemode=UseOutlines,
colorlinks        = true,
linkcolor         = blue,
urlcolor          = magenta,
citecolor         =red,
backref           = page
]{hyperref}
\usepackage[draft,inline,nomargin,index]{fixme}
\fxsetup{theme=color,mode=multiuser}
\FXRegisterAuthor{ad}{aad}{\color{purple}AD}
\FXRegisterAuthor{jl}{aajl}{\color{blue}JL}
\FXRegisterAuthor{flm}{aaflm}{\color{orange}FLM}

\usepackage{fourier-orns}
\usepackage{mathrsfs}
\DeclareMathOperator{\Aut}{Aut}
\DeclareMathOperator{\Stab}{Stab}

\DeclareMathOperator{\supp}{supp}
\renewcommand{\d}{\mathrm d}

\renewcommand{\leq}{\leqslant}
\renewcommand{\geq}{\geqslant}
\renewcommand{\le}{\leqslant}
\renewcommand{\ge}{\geqslant}
\renewcommand{\epsilon}{\varepsilon}
\renewcommand{\phi}{\varphi}
\renewcommand{\rm}{\mathrm}
\newcommand{\m}{\mathbf}
\newcommand{\id}{\rm{id}}

\newcommand{\mc}{\mathcal}
\newcommand{\inv}{^{-1}}
\newcommand{\cfg}[1]{[#1]_{\rm{com}}}
\newcommand{\dcfg}[1]{\overline{D(\cfg{#1})}}
\newcommand{\Sym}{\rm{Sym}}
\newcommand{\szn}{\Sym(\m Z,\m N)}
\newcommand{\tendv}{\xrightarrow[n\rightarrow +\infty]{}}
\newcommand{\ord}{\rm{ord}}

\theoremstyle{definition}
\newtheorem{dfn}{Definition}[section]

\theoremstyle{plain}
\newtheorem{theo}[dfn]{Theorem} 
\newtheorem{pro}[dfn]{Proposition}
\newtheorem{lem}[dfn]{Lemma}
\newtheorem{cor}[dfn]{Corollary}

\newtheorem{theointro}{Theorem}

\newtheorem{prointro}[theointro]{Proposition}

\theoremstyle{remark}
\newtheorem{rem}[dfn]{Remark}
\newtheorem{ex}[dfn]{Example}

\title{On the commensurating full group}
\author{Antoine Derimay}
\date{}

\begin{document}
\begin{abstract}
    We introduce a new Polish group, called the commensurating full group, associated to an ergodic measure-class preserving transformation of a standard atomless probability space. It is an analogue of the $\rm L^1$ full group defined by Le Maître in \cite{LM18} which does not need the transformation to preserve a measure to be defined. We prove, among others, that it is a complete invariant of flip conjugacy, and that it is quasi-isometric to the line in the sense of Rosendal.
\end{abstract}
\maketitle
\tableofcontents
\section{Introduction}
Let $(X,\mu)$ be a standard atomless probability space. The question of conjugacy of two measure preserving automorphisms has been one of the most studied problems in ergodic theory, notorious for being very hard to understand in the general case. Recently even, it has been proven that the conjugacy equivalence relation is not Borel (\cite{FRW11}). Some conjugacy invariants have been constructed, like entropy, which even turns out to be a complete conjugacy invariant for Bernoulli shifts. This shows that in some specific cases, the situation is much more manageable. 

Another invariant is the spectrum of the associated Koopman operator $f\mapsto f\circ T$ on $\rm L^2(X,\mu)$, which is a subgroup of the circle $\m S^1\subset \m C$.
We say $T$ has \emph{discrete spectrum} if there is a Hilbert basis of $\rm L^2(X,\mu)$ comprised of eigenfunctions of that operator. If $T$ has discrete spectrum, the spectrum  determines the conjugacy class of $T$ by the Halmos-von Neumann Theorem, see \cite{Hal56}.

There has thus been a common usage of associating groups to dynamical systems, like for instance the commutant $C(T)$ in $\Aut(X,\mu)$ of an ergodic transformation. A notable result in that direction is that of King \cite{Kin86}, proving that the commutant of a rank-one ergodic transformation is the weak closure of its powers. This group can be huge however, in the sense that it is not locally compact in general. We will thus mostly want to associate \emph{Polish} groups to ergodic transformations.

Another such example is that of the full group $[T]$ of $T$, made of all the transformations $S\in\Aut(X,\mu)$ whose orbits are contained in those of $T$. It is a natural invariant of conjugacy, and in fact a complete invariant of orbit equivalence for ergodic transformations, where two transformations are said to be orbit equivalent if up to conjugating, they have the same orbits. However, it turns out this Polish group is not the most useful as a complete invariant, since any two ergodic pmp transformations are in fact orbit equivalent \cite{Dye59}, so that all full groups are isomorphic. Hence we need a finer topological group in order to have more interesting invariants.\medskip 

There are two main ways of seeing elements in the full group $[T]$. One is with the \emph{cocycle}, defined for any $S\in [T]$ by 
$$S(x)=T^{c_S(x)}(x),$$
which is well-defined almost everywhere as long as $T$ is \emph{aperiodic}, that is, almost every orbit of $T$ is infinite. (Note that ergodicity implies aperiodicity.) One may then consider integrability conditions of the cocycle, which allows us to define the $\rm L^1$-full group $[T]_1$ of $T$ \cite{LM18}, or the $\phi$-integrable full group (see \cite{CJMT23}). These groups are naturally Polish, and $[T]_1$ turns out to be a complete invariant of \emph{flip-conjugacy} ($S$ and $T$ are said to be flip-conjugate if $S$ is conjugate to  either $T$ or $T\inv$).

The other way of seeing elements of the full group is by their action on the orbits of $T$: since $S\in[T]$ preserves the $T$-orbits, it induces an action on (almost) every $T$-orbit, which are all isomorphic to $\m Z$.
One may then consider new subgroups of $[T]$ like the \emph{commensurating full group} $\cfg T$ of $T$, which asks that $S$ ``preserves" the positive $T$ orbits, that is, $|T(S^\m N(x)\triangle S^\m N(x)|<\infty$ for $\mu$-a.e. $x$. This group will be the main focus of this paper. Note that, unlike the previous subgroups of $[T]$, the definition of the commensurating full group still makes sense if, instead of a measure preserving automorphism $T$, we are only given a measure-class preserving automorphism, as we only need to have a definition of what ``almost everywhere" means. One of our main results is the following (see Theorem \ref{theo: cfg T invar flip conj}):

\begin{theointro}\label{theointro: cfg complete invt of flip conj}
    Let $T_1,T_2\in\Aut(X,[\mu])$ be ergodic transformations. The following are equivalent:
    \begin{enumerate}
        \item $\cfg {T_1}$ and $\cfg{T_2}$ are isomorphic as abstract groups,
        \item $\cfg {T_1}$ and $\cfg{T_2}$ are isomorphic as topological groups,
        \item $T_1$ and $T_2$ are flip-conjugate.
    \end{enumerate}
\end{theointro}

Hence the commensurating full group is a complete invariant of flip conjugacy, both as a topological (Polish) group, and as an abstract group, which extends the work of Le Maître (\cite{LM18},\cite{LM21}) to the measure-class preserving case. We now give a few words on the proof of Theorem \ref{theointro: cfg complete invt of flip conj}.\medskip

In the integrable case, the fact that two ergodic transformations $T_1,T_2$ with $T_2\in[T_1]_1$ should be flip-conjugate was originally proven by Belinskaya in \cite{Bel68}, although without the definition of the $\rm L^1$ full group, only with integrable cocycles. A new proof of this result was then given by Katznelson in unpublished lecture notes, which are summed up in \cite[Appendix A]{CJMT23}. This new proof is of interest to us because of its structure. It is essentially comprised of two parts, the first one where it is proven, albeit with a different vocabulary, that $[T]_1<\cfg T$, and the second one proving that if $T_i$ are ergodic, have the same orbits, and $T_2\in\cfg {T_1}$, $T_1$ and $T_2$ are flip conjugated by constructing an explicit conjugating map. This essentially gives us a part of Theorem \ref{theointro: cfg complete invt of flip conj}. The other comes from a result of Fremlin (Theorem \ref{theo:Fremlin}), stating that any abstract isomorphism between two "big enough" subgroups of $\Aut(X,[\mu])$ actually comes from a conjugation, completing the proof of Theorem \ref{theointro: cfg complete invt of flip conj}.\medskip

We now come back to the definition of $\cfg T$. An element $S$ of $[T]$ induces a permutation of the orbits of $T$, which are all isomorphic to $\m Z$ under the orbit map. We thus get a (continuous) embedding from $[T]$ to the group $\rm L^0(X,\mu,\Sym(\m Z))$ of measurable maps from $X$ to the group of permutations $\Sym(\m Z)$ of $\m Z$, where multiplication of two elements happens pointwise. As argued before, we are interested in defining subgroups of $[T]$ which might keep more information than $[T]$ itself. 
One way to do this is to choose a Polish subgroup $G$ of $\Sym(\m Z)$, and to only look at elements of $[T]$ whose image through the embedding lies in $\rm L^0(X,\mu,G)$. This is what is done here, with $G=\szn$, the \emph{commensurating group of} $\m Z$. It is the group of permutations $\sigma\in \Sym(\m Z)$ for which $$|\sigma(\m N)\triangle \m N|<\infty.$$
It can be identified with the group of homeomorphisms of the two-point compactification $\overline{\m Z}$ of $\m Z$ which preserves $+\infty$. It is then a Polish group for the compact open topology on $\rm{Homeo}(\overline{\m Z})$, see \cite[Section 4.A]{Cor16} which in turn makes $\cfg T$ into a Polish group.

The group $\szn$ admits a continuous surjective morphism $I:\szn\to\m Z$ called the index map. It counts the number of orbits of $\sigma$ going from $-\infty$ to $+\infty$ minus the number of orbits going from $+\infty$ to $-\infty$ (both these numbers are finite because $\sigma\in\szn$). Since it is invariant under conjugation by the shift in $\m Z$, ergodicity of $T$ means that the index map associated to a given permutation of a $T$-orbit given by $S$  is in fact constant, yielding a continuous morphism $I:\cfg T\to\m Z$.

Many of the properties ascribed to $[T]_1$ and $I$ in \cite{LM18} still hold in this context, since although the definition of $I$ there is completely different, they turn out to coincide. We thus get descriptions of generating sets of $\cfg T$ for instance. A nice result of \cite{LM18} which is still true in the commensurating context is the following (see \cite[Corollary 4.20]{LM18} and Theorem \ref{theo: plein de sgs égaux au dérivé}):

\begin{theointro}\label{theointro: plein de sgs égaux au dérivé}
    Let $T\in\Aut(X,[\mu])$ be ergodic. The following normal subgroups of $\cfg T$ are all equal:
    \begin{enumerate}
        \item The subgroup generated by the periodic elements,
        \item The subgroup generated by the involutions,
        \item The closure of the derived subgroup,
        \item The kernel of the index map.
    \end{enumerate}
\end{theointro}

Hence the topological derived subgroup $\dcfg T$ is an interesting group to study, as it contains "a lot" of the elements of $\cfg T$. In fact, it is possible from $\dcfg T$ to reconstruct $\cfg T$ as a whole, meaning that $\dcfg T$ is itself a complete invariant of flip-conjugacy (see Theorem \ref{theo: dcfg T invar of flip conj}).\medskip

For (some) Polish groups, there is a notion of quasi-isometry, in that there is a maximal class of norms on $G$ which are quasi-isometric to each other, see \cite{Ros22} or Definition \ref{dfn: maximal metric}. We determine here that every norm on $\dcfg T$ is bounded, meaning that $\dcfg T$ is quasi-isometric to a point. This in turn lets us prove that $\cfg T$ is quasi-isometric to $\m Z$, with the index map being one such quasi-isometry.

\begin{theointro}\label{theointro: I est une QI}
    The index map 
    $$I:\cfg T\to \m Z$$
    is a quasi-isometry.
\end{theointro}

Combining this with the previous result, we get that $\dcfg T$ is a complete invariant of flip conjugacy, which is a \emph{bounded} Polish group.

This behaviour is quite different from that of the $\rm L^1$-full group, as in that case the maximal norm is given by 
$$\int_X|c_S|\d\mu,$$
which is very much unbounded even on $\overline{D([T]_1)}$.\medskip

Finally, we study topological simplicity results of the topological derived subgroup. Since $\cfg T$ admits a continuous surjective morphism to $\m Z$, it cannot satisfy any simplicity property. The kernel of this morphism however could satisfy such properties, and topological simplicity of the $\rm L^1$ full group was in fact proven in \cite{LM18}. We prove the same result in the case of the commensurating full group (see Theorem \ref{theo: dcfg is top simple}):

\begin{theointro}\label{theointro: dcfg is top simple}
    For any ergodic $T\in\Aut(X,[\mu])$, $\dcfg T$ is topologically simple.
\end{theointro}

Algebraic simplicity of $\overline{D([T]_1)}$ was disproved in \cite{LM18} however, using non-integrability of return times of a Baire-generic subset of $\rm{MAlg}(X,\mu)$. This approach cannot work anymore since integrability does not make sense in the measure-class preserving case in the first place. Furthermore, we prove that this non-transitivity of the Boolean action does not hold when considering the commensurating full group, although only in the type $\mathrm{II}$ case (see Proposition \ref{pro: transi de cfg en pmp}).

\begin{prointro}\label{prointro: transi de cfg en type II}
    Let $T\in\Aut(X,[\mu])$ be ergodic, and assume there exists a $\sigma$-finite measure $\nu\sim\mu$ which is preserved by $T$. If $A,B$ are Borel subsets such that $\nu(A)<\nu(B)$, there is an element $S\in\cfg T$ such that $S(A)\subset B$.
\end{prointro}

In particular, in the type probability measure-preserving case any measurable subset $A\subset X$ of measure $<\frac 12$ can be sent to a set that is disjoint of itself. This is in complete opposition with $[T]_1$ (see \cite[Corollary 4.23]{LM18}).

Hence the question of algebraic simplicity of the derived commensurating full group remains unknown, and it is unclear which answer one should expect.\medskip

\noindent\textbf{Outline of the paper.}
In Section \ref{sec: Defs+props}, we define the relevant objects of study, and prove some general facts, for instance that $\cfg T$ is Polish. In Section \ref{sec: Basic results}, we start by proving a number of decomposition results for elements of $\cfg T$, and use them to prove Theorem \ref{theointro: plein de sgs égaux au dérivé}. We also prove Proposition \ref{prointro: transi de cfg en type II}. In Section \ref{ssec: dcfg bounded} we prove the fact that $\dcfg T$ is bounded and Theorem \ref{theointro: I est une QI}, and in Section \ref{ssec: complete invts of flip conj} we prove Theorem \ref{theointro: cfg complete invt of flip conj}. Finally, in Section \ref{sec: Topo simp} we prove Theorem \ref{theointro: dcfg is top simple}.\medskip

\noindent\textbf{Acknowledgments.} The author thanks François Le Maître for suggesting the subject of this paper, as well as tirelessly answering \emph{numerous} questions.\medskip

\section{Definitions and first properties}\label{sec: Defs+props}
\subsection{Polish groups and the commensurating group}
\begin{dfn}
    A topological space $X$ is said to be \emph{Polish} if it is completely metrizable (i.e. it admits a complete metric compatible with the topology of $X$), and if it admits a countable dense subset.
\end{dfn}

For instance, any compact or locally compact space is Polish, countable products of Polish spaces with the product topology are still Polish. A subspace $Y$ of a Polish space $X$ is Polish for the induced topology if and only if $Y$ is a $G_\delta$ of $X$.
A Polish group is a topological group such that the underlying space is Polish. The family of Polish groups includes countable discrete groups, locally compact groups, but also the group $\Sym(\m Z)$ of all permutations of $\m Z$, along with the topology of pointwise convergence.\medskip

There is a notion of quasi-isometry for Polish groups, promoted by Rosendal (see for instance \cite{Ros22}), which is based on the notion of maximal distances.

\begin{dfn}\label{dfn: maximal metric}
    A compatible left-invariant metric $d$ on a Polish group $G$ is said to be \emph{maximal} if for any other compatible (left invariant) distance $d'$ on $G$, there is a constant $C>0$ such that 
    $$d'(g,h)\le C d(g,h)+C$$
    for all $g,h\in G$.
\end{dfn}
All maximal distances are quasi-isometric to one another, so this defines a natural quasi-isometry class for $G$, if a maximal distance exists. This is not always the case, and the interested reader is referred to \cite{Ros22} for details.

We will for the most part only need the following related definition:

\begin{dfn}
    A Polish group is said to be \emph{bounded} if some (equivalently any) bounded distance on $G$ is maximal. Equivalently, $G$ is bounded if for any neighborhood $V$ of $1$ in $G$, there is $n\in \m N$ such that 
    $$V^n=G.$$
\end{dfn}

We now go towards defining the commensurating group.

\begin{dfn}
    Let $\sigma\in\Sym(\m Z)$. We let $\mc L(\sigma)$ be the \emph{length} of $\sigma$, defined by
    $$\mc L(\sigma):=|\sigma(\m N)\triangle \m N|\in\m N\cup\{\infty\},$$
    where $A\triangle B=(A\cup B)\setminus (A\cap B)=(A\setminus B)\sqcup (B\setminus A)$ is the symmetric difference of the sets $A$ and $B$.
    
    We denote by $\szn$ the group of all permutations of $\m Z$ for which $\mc L(\sigma)<\infty$.
\end{dfn}

We may define a topology on $\szn$ in the following way:

\begin{dfn}
    Let $\mathrm{Comm}(\m N)$ be the set of subsets of $\m Z$ commensurated to $\m N$ (i.e. $|\m N\triangle M|<+\infty$). Then $\szn$ acts faithfully on $\mathrm{Comm}(\m N)$, and we endow $\szn$ with the group topology induced by the inclusion in $\Sym(\mathrm{Comm}(\m Z))$. 
    
    In other words, a basis of neighborhoods of $1$ in $\szn$ is given by the pointwise stabilizers $H(F)$ of a finite subset $F$ of $\mathrm{Comm}(\m N)$.
\end{dfn}

The following is a rewriting of \cite[Lemma 4.A.1]{Cor16}, but it will give us a good intuition for some later results.

\begin{lem}\label{lem: caracterisation convergence szn}
    Let $(\sigma_n)_n$ be a sequence in $\szn$. Then $\sigma_n\to 1$ in $\szn$ if and only if $\sigma_n\to 1$ in $\Sym(\m Z)$ and $\sigma_n(\m N)=\m N$ for all $n$ large enough.
\end{lem}

\begin{proof}
    We keep the notations used in \cite[4.A]{Cor16}, namely let $H_\m N(F)$ be the pointwise stabilizer of a finite set $F\subset \m Z$ intersected with the stabilizer of $\m N$.
    
    By definition, $\sigma_n\in H_\m N(\emptyset)=\Stab(\m N)$ for large $n$, and by Lemma 4.A.1 of \cite{Cor16}, if $\sigma_n\to1$ in $\szn,$ then $\sigma_n\to1$ in $\Sym(\m Z)$. 
    
    Conversely, if $\sigma_n\to1$ in $\Sym(\m Z)$ and $\sigma_n\in \Stab(\m N)$ for $n$ large enough, then for any finite $F\subset \m Z,\:\sigma_n\in H(F)$ for $n$ large enough. Since $\sigma_n\in \Stab(\m N)$ for $n$ large enough, $\sigma_n\in H_\m N(F)=H(F)\cap \Stab(\m N)$ for large $n$, and we conclude again by Lemma 4.A.1 of \cite{Cor16}.
\end{proof}

\begin{pro}
    $\szn$ along with this topology is a Polish group.
\end{pro}

\begin{proof}
    Note first of all that the involution $\Phi:A\mapsto A\triangle \m N$ swaps $\mathrm{Comm} (\m N)$ and the set $\mc P_f(\m Z)$ of finite subsets of $\m Z$, so that $\mathrm{Comm}(\m N)$ is countable. In particular, $\Sym(\mathrm{Comm}(\m N))$ is Polish. We now prove that $\szn$ is closed in $\Sym(\mathrm{Comm}(\m N))$.

    Let $g\in \Sym(\mathrm{Comm}(\m N))$. We conjugate the action of $g$ on $\mathrm{Comm}(\m N)$ through $\Phi$ into an action on $\mc P_f(\m Z)$ in the following way: for a finite subset $A$ of $\m Z$, let 
    $$g\cdot A:=g(A\triangle \m N)\triangle g(\m N).$$
    With this notation in hand, if we manage to prove that 
    \begin{eqnarray*}
        \szn &=&\{g\in\Sym(\mathrm{Comm}(\m N)), \forall A,B\in\mc P_f(\m Z), |g\cdot A|=|A|, \\
        &&g\cdot(A\cup B)=g\cdot A\cup g\cdot B\text{ and } \bigcup_{n\in\m N}g\cdot\llbracket1,n\rrbracket=g(\m N)\}
    \end{eqnarray*}
    we would be done, as all of these conditions clearly are closed in $\Sym(\mathrm{Comm}(\m N))$.\medskip 

    The fact that $LHS\subset RHS$ is immediate, as for $g\in\szn$, $g\cdot A=g(A)$. Let now $g\in RHS$. For $k\in \m Z$, let $\sigma(k)$ be such that $g\cdot\{k\}=\{\sigma(k)\}$. The map $\sigma:\m Z\to\m Z$ has to be injective since $g$ is, and 
    $$g\cdot[g\inv (g(\m N)\triangle \{n\})\triangle \m N]=\{n\}$$
    so $\sigma\in\Sym(\m Z)$. The conditions directly imply that for $A\subset \m Z$ finite $g\cdot A=\sigma(A)$, and $\sigma(\m N)=g(\m N)$. Finally, by definition of $g\cdot A$ for finite $A$,
    $$g(A\triangle \m N)=g\cdot A\triangle \m N=\sigma(A)\triangle \sigma(\m N)=\sigma(A\triangle \m N).$$
    Hence the action of $g$ on $\rm{Comm}(\m N)$ comes from an action of some $\sigma\in\Sym(\m Z)$, and $\sigma\in \szn$ since the action preserves $\rm{Comm}(\m N)$.
\end{proof}

\begin{dfn}
    Let $\sigma\in\szn.$ We define the \emph{index map} $I:\szn\to\m Z$ as
    $$I(\sigma):=|\m N\setminus \sigma\m N|-|\sigma\m N\setminus \m N|.$$
\end{dfn}
Note that replacing the $-$ in the above definition by a $+$ gives $\mc L(\sigma)$ which is finite by definition. In particular, the above quantity is well defined.

It can be easily checked that $I$ is a continuous surjective morphism from $\szn$ to $\m Z$.\medskip

The following is not useful for any proofs later on, but might be a more visual description of the index map.

\begin{lem}
    $I(\sigma)$ is also the number of $\sigma$-orbits going from $-\infty$ to $+\infty$, minus the number of $\sigma$-orbits going from $+\infty$ to $-\infty$.
\end{lem}

\begin{proof}
    This result comes from the observation that periodic orbits cross the origin the same amount of times in both directions, and the same holds for infinite orbits that go from $\pm\infty$ to $\pm\infty$.
\end{proof}

\subsection{Ergodic theory and spaces of measurable maps}
This section was heavily inspired by \cite{LM18}, as well as \cite{KLM15}. More details concerning those topics can be found there.
Let $(X,\mu)$ be a standard probability space, i.e. a standard Borel space equipped with a Borel non-atomic probability measure $\mu$. All such spaces are measure isomorphic to the interval $[0,1]$ equipped with the Lebesgue measure. Throughout the paper, $(X,\mu)$ will always denote a standard probability space.

$\Aut(X,\mu)$ denotes the group of automorphisms of $(X,\mu)$ up to measure $0$, that is the set of measurable bijections $T:X\to X$ such that $T_*\mu=\mu$, where $T_*\mu$ is defined by $T_*\mu(A)=\mu(T\inv A)$ for any Borel set $A\subset X$, where two functions are identified if they coincide up to a set of measure $0$. It is a Polish group for the uniform topology, the topology generated by the uniform distance
$$d_u(S,T)=\mu\{x\in X,S(x)=T(x)\}.$$

\begin{dfn}
    Two Borel $\sigma$-finite measures $\mu,\nu$ are said to be \emph{equivalent} if they have the same null sets, that is for any $A\subset X$, $\mu(A)=0\iff\nu(A)=0$, which we write $\mu\sim\nu$.
    
    It is easily seen that $\sim$ is an equivalence relation, and we denote by $[\mu]$ the class of $\mu$, which is called the \emph{measure class} of $\mu$.
\end{dfn}

Note that the Radon-Nikodym theorem states that if $\mu\sim\nu$, there is a measurable function $f:X\to \m R_{>0}$ such that for any measurable subset $A\subset X$, we have
$$\nu(A)=\int_Af\d\mu.$$
$f$ is unique up to measure $0$, and is called the Radon-Nikodym derivative of $\nu$ with respect to $\mu$, and is denoted $\dfrac{\d\nu}{\d\mu}$.

\begin{dfn}
    We denote by $\Aut(X,[\mu])$ the group of \emph{measure-class preserving automorphisms} of $(X,\mu)$, that is the group of bijections $T:X\to X$ such that $T_*\mu\sim\mu$, where two maps are identified if they coincide almost everywhere.
\end{dfn}

\begin{dfn}
    An automorphism $T\in \Aut(X,[\mu])$ is said to be \emph{ergodic} if any $T$-invariant Borel subset $A\subset X$ is either null or conull.
\end{dfn}

Note that the notion of ergodicity depends only on the measure class $[\mu]$, and not on $\mu$ itself.\medskip

For a subset $A\subset X$ and $T\in\Aut(X,[\mu])$ such that any $T$-orbit which encounters $A$ encounters it infinitely often in both directions, we may define the \emph{induced map} $T_A\in\Aut(X,[\mu])$ supported on $A$ in the following way.

For $x\in A$, let $n_{A,T}(x)$ be the \emph{return time} of $x$ to $A$, i.e. the smallest $n>0$ such that $T^n(x)\in A$, which exists (a.e.) by assumption. The map $T_A$ is then defined by
$$T_A(x)=T^{n_{A,T}(x)}(x).$$
Note that the hypothesis on $T$ and $A$ is automatically verified if $A$ is $T$-invariant, or if $T$ is ergodic.

\begin{dfn}
    Let $Y$ be a Polish space. We let $\mathrm L^0(X,\mu,Y)$ be the space of measurable maps from $X$ to $Y$, where two maps which coincide almost everywhere are identified. We will write $\mathrm L^0(X,Y)$ if the measure class is unambiguous.
    
    This space can be endowed with the topology of convergence in measure, which makes $\mathrm L^0(X,\mu,Y)$ into a Polish space. It has the following basis of neighborhoods:
    $$V_\epsilon(f)=\{g\in \mathrm L^0(X,\mu,Y),\mu\{x\in X,d_Y(f(x),g(x)< \epsilon\}>1-\epsilon\}$$
    where $d_Y$ is a distance compatible with the topology on $Y$. The topology it generates does not depend on the choice of $d_Y$, nor on the choice of the measure $\mu$ in the class $[\mu]$ (see \cite{Moo76}).
\end{dfn}

Moreover, if $Y=G$ happens to be a Polish group, then $\mathrm L^0(X,\mu,G)$ is a group where multiplication is just pointwise, and the topology defined previously makes it a Polish group.

\subsection{The commensurating full group}

\begin{dfn}
    Let $T\in\Aut(X,[\mu])$. We denote by $[T]$ the \emph{full group} of $T$, that is the subgroup comprised of the maps $S\in \Aut(X,[\mu])$ such that for a.e. $x\in X$, 
    $$S(x)=T^{c_S(x)}(x)$$
    for some integer $c_S(x)$. Any (measurable) map $c_S:X\to \m Z$ which verifies the above condition is called a \emph{cocycle} of $S$ with respect to $T$. It is unique if the action of $T$ on $X$ is essentially free, for instance when $T$ is ergodic.
\end{dfn}
We also define the pseudo-full group $[[T]]$ of  $T$ as the groupoid of partial maps 
$$\{S:A\to B, A,B\subset X, \forall x\in A,S(x)\in T^\m Z(x)\},$$
where we again identify maps which coincide almost everywhere.

This definition will be useful on a couple of occasions.

Note that if $T$ preserves some $\sigma$-finite measure $\nu\sim\mu$, then so does any $S\in[T]$. Indeed, if we let $A_n=\{x\in X, S(x)=T^n(x)\}$, then the $A_n,n\in\m Z$ form a partition of $X$, so that 
$$\nu(S\inv(B))=\sum_{n\in\m Z}\nu(S\inv(B\cap A_n))=\sum_{n\in\m Z}\nu(T^{-n}(B\cap A_n))=\sum_{n\in\m Z}\nu(B\cap A_n)=\nu(B).$$

\begin{rem}\label{rem: general def of full group}
    The previous reasoning leads us to the general notion of a full group. A group $\m G\le \Aut(X,[\mu])$ is said to be \emph{full} if it is stable under countable cutting and pasting. In other words, $\m G$ is full if for any countable partition $$X=\bigsqcup_{i\in I} A_i$$  and any $\{g_i\in G,i\in I\}$ such that the $(g_iA_i)_{i\in I}$ still form a partition of $X$, the element $g\in \Aut(X,[\mu])$ defined by the reunion of the $g_i|_{A_i}$ is still in $\m G$. 
    
    The full group of $T$ could thus also be defined as the smallest full group containing $T$.
\end{rem}

Recently, F. Le Maître introduced in \cite{LM18} a variant of the full group, called the integrable full group, or $\mathrm L^1$-full group. It is defined as follows:

\begin{dfn}
    Let $T\in\Aut(X,\mu)$ be aperiodic. The $\mathrm L^1$\emph{-full group} of $T$ is the group $[T]_1$ of automorphisms $S\in [T]$ such that 
    $$\int_X |c_S(x)|\d \mu<\infty.$$
\end{dfn}
The $\mathrm L^p$-full group $[T]_p$ is defined in the same way, it is the group of $S\in[T]$ such that $c_S\in \mathrm L^p(X)$.\medskip

Note that this definition depends on $\mu$ itself, as integrability is a notion that varies inside a measure class, (with the exception of the case $p=\infty$). In particular, all the results Le Maître obtained on $[T]_1$ are only valid for pmp maps.

The following construction turns out to be closely related to $[T]_1$, and so generalizes it to any $T\in \Aut(X,[\mu])$.\medskip

Note that when $T$ is aperiodic, every orbit is in bijection with $\m Z$, and so since any $S\in [T]$ preserves the $T$-orbits, they induce a permutation of $\m Z$. We are interested in those whose induced permutations commensurate $\m N$. More precisely, let $\Phi:[T]\to \mathrm L^0(X,\Sym(\m Z))$ be defined by
$$\Phi(S)(x)(n)=n+c_S(T^n(x)).$$
By \cite[Proposition 13]{KLM15}, $\Phi$ is a topological group embedding. 
\begin{dfn}\label{dfn: cfg}
We let 
$$\cfg T:=\Phi\inv(\mathrm L^0(X,\szn)).$$
We endow $\cfg T$,with the pullback of the topology of $\mathrm L^0(X,\szn)$. This group is called the \emph{commensurating full group} of $T$, and will be the main object of study in this paper.
\end{dfn}

The main interest of $[T]_1$ was Beliskaya's Theorem (see \cite{Bel68}, \cite{LM18}, \cite{CJMT23}):

\begin{theo}
    If $T_1,T_2\in\Aut(X,\mu)$ are ergodic such that $T_1\in[T_2]_1$ and if $T_1$ and $T_2$ have the same orbits, then $T_1$ and $T_2$ are flip conjugate: either $T_1$ and $T_2$ are conjugate, or $T_1$ and $T_2\inv$ are conjugate.
\end{theo}

In fact an alternative proof of Belinskaya's result was later provided by Katznelson (in unpublished course notes, see \cite{CJMT23}), and this proof goes in two steps, first proving that $[T]_1<\cfg T$ and then proving the same result as Belinskaya where the assumption $T_1\in[T_2]_1$ has been replaced by $T_1\in\cfg{T_2}$. The details of this proof can be found in \cite[Appendix A]{CJMT23}.\medskip

From this we deduce that
\begin{pro}
    For any $p\in[1,\infty],$ we have 
    $$[T]_p<[T]_1<\cfg T.$$
\end{pro}

We now give a couple basic properties of $\cfg T$.

\begin{pro}
    $\cfg T$ is a Polish group.
\end{pro}

\begin{proof}
    We will prove that $\Phi(\cfg T)$ is closed in $\mathrm L^0(X,\szn)$. To that end, consider the following commutative diagram made of injections:

    \[\begin{tikzcd}
	\cfg T && \mathrm L^0(X,\szn) \\
	\\
	\: [T] && \mathrm L^0(X,\Sym(\m Z))
	\arrow["{\Phi}", hook, from=1-1, to=1-3]
	\arrow[hook', from=1-1, to=3-1]
	\arrow[hook', from=1-3, to=3-3]
	\arrow["{\Phi}", hook, from=3-1, to=3-3]
    \end{tikzcd}\]
    We know $\Phi([T])$ is closed in $\mathrm L^0(X,\Sym(\m Z))$, and that the topology of $\mathrm L^0(X,\szn)$ is finer than that of $\mathrm L^0(X,\Sym(\m Z))$. Hence $\Phi([T])\cap \mathrm L^0(X,\szn)=\Phi(\cfg T)$ is closed in $\mathrm L^0(X,\szn)$.
\end{proof}

Let $S\in\cfg T$, and $x\in X$. Note that $I(\Phi(S)(x))=I(\Phi(S)(T(x)))$, and since $I:\szn\to\m Z$ is continuous, $x\mapsto I(\Phi(S)(x))$ is measurable. Hence by ergodicity of $T$, $I(\Phi(S)(x))$ does not depend on $x$.

\begin{dfn}
    Let $I(S)$ be this common value. The map $I$ is called the \emph{index map} $I:\cfg T\to\m Z$, which is easily checked to be a continuous morphism.
\end{dfn}

Note that by \cite[Proposition 4.17]{LM18}, this definition coincides with the one of \cite[Definition 4.3]{LM18} on $[T]_1$.\medskip

The following notions were first defined in \cite{LMS23} in the context of pmp actions, but they can just as well be defined in the measure-class preserving case, and they will turn out to be useful on a variety of occasions throughout the paper.

\begin{dfn}
    A subgroup $\m G\le \Aut(X,[\mu])$ is said to be \emph{finitely full} if for any finite partition $X=A_1\sqcup\cdots \sqcup A_n$ and any $g_1,\ldots g_n\in\m G$ such that the sets $g_iA_i$ still form a partition of $X$, the element $g\in\Aut(X,[\mu])$ obtained as the reunion of the $g_i|_{A_i}$ is still an element of $\m G$.
\end{dfn}
(See Remark \ref{rem: general def of full group} for the origin of the name.)

\begin{dfn}
    A finitely full group $\m G$ is said to be \emph{induction friendly} if it is stable under taking induced transformations whenever allowed to and, if $U \in \m G$ and $(A_n)_n$ is an increasing sequence of $U$-invariant sets such that $\bigcup_nA_n = A$, then $U_{A_n}\to U_A$.
\end{dfn}

\begin{dfn}
    A subgroup $\m G\le \Aut(X,[\mu])$ is said to be \emph{aperiodic} if it admits a countable subgroup whose action on $X$ has no finite orbits.
\end{dfn}

Of course, those definitions would not be the most useful if they did not apply to our object of study:

\begin{pro}
    $\cfg T$ is a finitely full induction friendly aperiodic group.
\end{pro}

\begin{proof}
    Aperiodicity is straightforward: pick a countable dense subset $A\subset \cfg T$. Then the subgroup generated by $A\cup\{T\}$ is countable, and any orbit under $\langle A\cup\{T\}\rangle$ contains the $T$-orbit, which is always infinite.\medskip

    Now onto finite fullness. Let $X=A_1\sqcup\cdots\sqcup A_n$ be a measurable partition of $X$, and $g_1,\ldots,g_n\in\cfg T$ be such that the $g_iA_i$ still partition $X$. As membership to $\cfg T$ is decided ($T$-) orbit by orbit and the $g_i$ preserve those orbits, we may restrict ourselves to one such orbit, which is the same as assuming $X=\m Z$.
    
    In that case, the $g_i$ each send finitely many positive numbers to a negative number, so it stays true for $\sigma$, and the same holds for sending negative numbers to positive numbers.\medskip

    We are now left with induction friendliness. We first prove stability under taking induced maps. Let $S\in\cfg T$, and $A\subset X$ be of positive measure. Note that for a.e. $x\in A$, there is some $n\ge 1$ for which $S^n(x)\in A$ by Poincaré recurrence. We now restrict to the study of the behaviour of $S_A$ on a single orbit of $T$. Denote by $\sigma$ the corresponding permutation of $\m Z$, and $A$ the intersection with this orbit. 
    
    Since $\sigma\in\szn$, there are finitely many orbits of $\sigma$ that have both positive and negative numbers. Among those, the finite orbits only add a finite perturbation, and can be disregarded. Take now an infinite orbit, say $\sigma^\m Z(x)$. Let $B$ be the set of $n\in\m Z$ for which $\sigma^n(x)\sigma^{n+1}(x)\le 0$. It is finite, let $m,M$ be its minimum and maximum respectively. Also, let $m'$ be the biggest element of $A$ smaller than $m$, and $M'$ the smallest element of $A$ bigger than $M$. Then as $\sigma$ is of constant sign outside of $[m,M]$, $\sigma_A$ is of constant sign outside of $[m',M']$. Hence this orbit also causes finite perturbations in the sign. A sum of finitely many finite numbers is finite, so $\sigma_A\in\szn$.\medskip

    Now let $S\in \cfg T$, and $A_1, A_2,\ldots$ be an increasing sequence of $S$-invariant sets. Let $A=\bigcup_i A_i$. We prove that convergence of $S_{A_i}$ to $S_A$ happens pointwise when embedding $\cfg T$ into $\mathrm L^0(X,\mu,\szn)$. This will imply the result as such a convergence is stronger than uniform convergence in measure. 
    
    We thus restrict to an orbit, which we identify with $\m Z$. Let $N\in\m N$, there are only finitely many orbits of $S$ which intersect $\llbracket -N,N\rrbracket$, say $S^\m Z(x_i),\:1\le i\le k$. Since all $A_i$ are unions of orbits, there is an integer $n_0$ such that 
    $$A_n\cap\left(\bigcup_{i=1}^kS^\m Z(x_i)\right)=A\cap\left(\bigcup_{i=1}^kS^\m Z(x_i)\right)$$
    for all $n\ge n_0$. Then $S_{A_n}$ coincides with $S_A$ on $\llbracket -N,N\rrbracket$, which is to say $S_{A_n}\to S_A$ in $\Sym(\m Z)$.
    
    Now let $N$ be big enough for every orbit that contains both positive and negative integers to intersect $\llbracket -N,N\rrbracket$. Note that for any $x$, $S_{A_n}(x)\in S^\m Z(x)$, and the same goes for $S_A$. As such, $S_AS_{A_n}\inv$ is supported away from the aforementioned orbits, and stabilizes $S$-orbits. In particular, it stabilizes $\m N$ for all $n\ge n_0$. By Lemma \ref{lem: caracterisation convergence szn}, $S_{A_n}$ does converge to $S$ in $\szn$, and we are done.
\end{proof}

\section{Basic results on the commensurating full group}\label{sec: Basic results}
\subsection{Structural results}
We start by proving some useful decomposition results for elements of $\cfg T$. For that, we will need the following lemma:

\begin{lem}
    Let $\sigma\in\szn$, and $x\in\m Z$. Then exactly one of the following holds:
    \begin{enumerate}
        \item $x$ is periodic ($\sigma^\m N(x)$ is finite),
        \item $\sigma^n(x)\tendv+\infty$,
        \item $\sigma^n(x)\tendv-\infty$.
    \end{enumerate}
\end{lem}

\begin{proof}
    These three cases are clearly mutually exclusive. Assume the orbit of $x$ is infinite, and that $(\sigma^n(x))_n$ does not verify $2$ or $3$. Let $x_n=\sigma^n(x)$. Since $x_n$ is injective (otherwise the orbit would be finite), 
    $$|x_n|\tendv +\infty.$$
    If neither $2$ or $3$ hold, we thus have an increasing sequence $k_n$ of integers such that $x_{k_n}>0$ and $\sigma(x_{k_n})=x_{k_n+1}<0$.
    
    But now the set $\{x_{k_n},n\in\m N\}$ is an infinite subset of $\m N$ which is sent to $\m Z^-$ by $\sigma$, contradicting the fact that $\sigma$ commensurates $\m N$.
\end{proof}

We can now transport this result to $\cfg T$ itself.
\begin{lem}
    Let $S\in\cfg T$. For $\mu$-a.e. $x$ in $X$, exactly one of the following holds:
    \begin{enumerate}
        \item $x$ is periodic
        \item $c_{S^n}(x)\tendv+\infty$ (we say $x$ is almost positive)
        \item $c_{S^n}(x)\tendv-\infty$ (we say $x$ is almost negative)
    \end{enumerate}
\end{lem}

\begin{proof}
    Via the embedding $$\Phi:\cfg T\hookrightarrow \mathrm L^0(X,\szn),$$
    we may identify the dynamics of $S$ on $T^\m Z(x)$ with that of $\Phi(S)(x)\in \szn$ on $\m Z$. The result then follows from the previous lemma.
\end{proof}

\begin{dfn}If an element $S\in\cfg T$ is such that any $x\in \mathrm{supp} S$ is almost positive, we say that $S$ is \emph{almost positive}. If every $x\in \mathrm{supp} S$ is almost negative, we say that $S$ is \emph{almost negative}.
\end{dfn}

\begin{dfn}
Let $T\in\Aut(X,\mu)$ be aperiodic. We define a partial order $\le_T$ on $X$ by $x\le_T y$ if $y=T^n(x)$ for some $n\geq 0$.
\end{dfn}

Note that by the previous lemma, if $S(x)\ge_T x$ for almost every $x\in\mathrm{supp}S,$ then $S$ is almost positive.

\begin{pro}\label{pro: decomp perio/+/-}
    Let $T \in \Aut(X, \mu)$ be aperiodic. Every element $S \in \cfg T$ can be written as a product $S = S_pS_+S_-$, where $S_p$ is periodic, $S_+$ is almost positive, $S_-$ is almost negative and these three elements have disjoint supports.
\end{pro}

\begin{proof}
Let $S\in\cfg T$. The set of $x\in X$ such that the sequence $(c_{S^k}(x))_{k\in\m N}$ is almost positive (resp. almost negative) is $S$-invariant. Call this set $A_+$ (resp. $A_-$). Finally let $A_p$ denote the set of $x\in X$ whose $S$-orbit is finite. Then $(A_-,A_+, A_p)$ is a partition of $X$ by the previous lemma, and since they are $S$-invariant we deduce that $S=S_{A_-}S_{A_+}S_{A_p}$. Clearly $S_p:=S_{A_p}$, $S_+:=S_{A_+}$ and $S_-:=S_{A_-}$ are as wanted.
\end{proof}

\begin{dfn}
    Let $T\in\Aut(X,\mu)$ be aperiodic. If $S\in\cfg T$, we say that $S$ is \emph{positive} if $S(x)\geq_T x$ for all $x\in X$. 
\end{dfn}

\begin{pro}\label{pro: almost positive is positive up to a periodic}Let $T\in\Aut(X,\mu)$ be aperiodic, and let $S\in\cfg T$ be almost positive.  Then there exists a positive element $S'\in\cfg T$ whose support is a subset of $\mathrm{supp} S$ such that $S{S'}\inv$ and ${S'}\inv S$ are periodic. 
\end{pro}
\begin{proof}
Consider the set $A:=\{x\in X: S^k(x)>_Tx\text{ for all }k>0\}$. Note that $T_A(x)$ may be defined as the first element greater than $x$ which belongs to $A$, so by the definition of $A$ we have $T_A(x)\leq_T S_A(x)$ for all $x\in A$. 

Let us show that $A$ intersects every non trivial $S$-orbit. Let $x\in \supp S$. Since $S$ is almost positive, the sequence $(c_{S^k}(x))_{k\in\m N}$ is almost positive. Let $k$ be the last natural integer such that $c_{S^k(x)}\leq 0$, then by definition $S^k(x)\in A$ as wanted.

Now for all $x\in A$ we have $S_A(x)\geq_TT_A(x)\geq_T x$, in particular $S_A$ is positive. Since $A$ intersects every $S$-orbit, we moreover have that $SS_A\inv$ and $S_A\inv S$ are periodic, so $S':=S_A$ is as desired. 
\end{proof}

\begin{lem}
Let $T\in\Aut(X,\mu)$ be aperiodic, let $S\in\cfg T$ be positive and let $A=\mathrm{supp} S$. Then $ST_A\inv$ is positive.
\end{lem}

\begin{proof}
Since $S$ is positive, we have $A= \{x\in X: S(x)>_Tx\}$. But then by the definition of $T_A$, for all $x\in A$ we have $S(x)\geq_T T_A(x)$. Since $A$ is $T_A$-invariant, we deduce that for all $x\in A$, $ST_A\inv(x)\geq_T x$. But the support of $ST_A\inv$ is contained in $A$ so the previous inequality actually holds for all $x\in X$, and we conclude that $ST_A\inv$ is positive. 
\end{proof}

\begin{pro}\label{pro: positive is product of induced}
Let $T\in\Aut(X,[\mu])$ be ergodic. Then every positive element is the product of finitely many transformations which are induced by $T$: for all $S\in \cfg T^+$ there exists $k\in\m N$ and $A_1,\ldots,A_k\subseteq X$ such that $$S=T_{A_1}\cdots T_{A_k}.$$
\end{pro}
\begin{proof}
Note that for $S\in\cfg T^+$, $I(S)\ge 0,$ with equality if and only if $S=\id_X$. We prove the result by induction on $I(S)$.

It is certainly true if $I(S)=0$. Assume now that every element $S\in\cfg T^+$ with $I(S)=n$ can be written as a product of transformations induced by $T$, let $S\in\cfg T^+$ with $I(S)=n+1$.

Let $A=\mathrm{supp} S$, by the previous lemma $ST_A\inv$ is positive. Since $T$ is ergodic, $A$ intersects almost every $T$-orbit, so that $I(T_A)=1$. We get $I(ST_A\inv)=n$, so that $ST_A\inv$ can be written as a product of transformations induced by $T$, and so does $S$.
\end{proof}

\begin{pro}\label{pro: cfg eng par perio & T}
    For an ergodic $T\in\Aut(X,[\mu])$, the group $\cfg T$ is generated by the periodic maps and $T$.
\end{pro}

\begin{proof}
    By Proposition \ref{pro: decomp perio/+/-}, $\cfg T$ is generated by periodic maps, together with almost positive and almost negative maps. Since the inverse of an almost negative map is almost positive, $\cfg T$ is generated by the periodic maps and the almost positive maps.
    
    By Propositions \ref{pro: almost positive is positive up to a periodic} and \ref{pro: positive is product of induced}, $\cfg T$ is generated by the periodic maps, along with maps of the form $T_A$ for some Borel subset $A$.
    
    Now let $A$ be a Borel subset of $X$. If it is of zero measure, $T_A=\id$ and we are done. Otherwise, by ergodicity, $A$ meets every $T$ orbit, so that $T(T_A)\inv$ is periodic, and we get the desired result.
\end{proof}

Recall that $[T]_\infty$ denotes the group of elements $S$ of $[T]$ for which their cocycle $c_S$ is essentially bounded.

\begin{lem}\label{lem: perio=lim bdd cocycles}
A periodic element of $\cfg T$ is a limit of elements of $[T]_\infty$.
\end{lem}

\begin{proof}
    Let $U\in\cfg T$ be periodic, and for $n\in \m N$ let $A_n:=\{x\in \supp U,\:\forall k\in\m Z,|c_{U^k}(U^l(x))|<n\}$ be the set of points whose $U$-orbits have diameter less than $n$ in the $T$-distance. Then $A_n$ is $U$-invariant and $\supp U=\bigcup_n A_n$, and $U_{A_n}\in[T]_\infty$. By induction friendliness, $U=\lim_n U_{A_n}$, and we are done.
\end{proof}

\begin{cor}
    Let $T\in\Aut(X,[\mu])$ be ergodic. Then $[T]_\infty$ is dense in $\cfg T$. In particular, if $T$ preserves $\mu$, $[T]_1$ is dense in $\cfg T$.
\end{cor}

\begin{proof}
    Let $G=\overline{[T]_\infty}$ be the closure of $[T]_\infty$ in $\cfg T$. We know by Lemma \ref{lem: perio=lim bdd cocycles} that any periodic $U\in\cfg T$ is in $G$, and we certainly have $T\in[T]_\infty$. Hence by the previous proposition, $G=\cfg T$.
\end{proof}

The following lemma indicates that the pseudo full groups are "as transitive" as they can be when acting on subsets of $X$, that is, the only possible constraint is that of an invariant measure.

\begin{lem}
    Let $T\in \Aut(X,[\mu])$, and let $\nu\sim \mu$ be a $T$-invariant measure. If $A,B$ are such that $\nu(A)=\nu(B)$, there is $\phi\in[[T]]$ such that $\mathrm{dom}(\phi)=A,\mathrm{rng}(\phi)=B$.
    
    If no such invariant measure exists, then for any $A,B$ of positive measure, there is $\phi\in[[T]]$ such that $\mathrm{dom}(\phi)=A,\mathrm{rng}(\phi)=B$.
\end{lem}

\begin{proof}
    The first part is \cite[Corollary of Lemma 5]{HIK74}, while the second part is \cite[Lemma 6]{HIK74}.
\end{proof}

\begin{lem}\label{lem: invo qui échange A et B}
    Let $T\in \Aut(X,[\mu])$, and let $\nu\sim \mu$ be a $T$-invariant measure. If $A,B$ are disjoint such that $\nu(A)=\nu(B)$, there is an involution $U\in[T]$ with $\supp U=A\sqcup B$ and $U(A)=B$.
    
    If no such invariant measure exists, then for any $A,B$ disjoint and of positive measure, there is an involution $U\in[T]$ with $\supp U=A\sqcup B$ and $U(A)=B$.
\end{lem}

\begin{proof}
    For both cases, we may apply the previous lemma to get $\phi\in [[T]]$ such that $\mathrm {dom}\:\phi=A,\mathrm{rng}\:\phi=B$. We then define $U$ by 
    $$U(x)=
    \begin{cases}
        \phi(x) \text{ if }x\in A,\\
        \phi\inv(x) \text{ if }x\in B,\\
        x \text{ otherwise.}
    \end{cases}$$
    It is easily seen that $U$ is as desired.
\end{proof}

\begin{lem}\label{lem: invo de support donné}
    Let $C$ be a Borel subset of $X$. Then there is an involution $U\in[T]$ such that $\supp U=C$.
\end{lem}

\begin{proof}
    If $T$ admits an invariant measure $\nu$, it is always possible to split $C$ into $C=A\sqcup B$ with $\nu(A)=\nu(B)=\nu(C)/2$ given that $X$ is standard.
    
    Otherwise, just split $C=A\sqcup B$ with $A,B$ of positive measure, which is again possible since $X$ is standard. 
    
    In either case, applying the previous lemma to $A$ and $B$ yields the desired result.
\end{proof}

\begin{lem}\label{lem: invos approximées par cocycle borné}
    Let $U\in [T]$ be an involution. Then there is an increasing family of $U$-invariant subsets $A_n$ of $\supp U$ such that $\supp U=\bigcup_n A_n$ and $U_{A_n}\in [T]_\infty$ for all $n$.
\end{lem}

\begin{proof}
    Let $A_n=\{x\in\supp U,|c_U(x)|<n\}$. We certainly have $\bigcup_nA_n=\supp U$, and clearly $U_{A_n}\in[T]_\infty$. Since $U$ is an involution, $A_n$ is $U$-invariant, so that $U_{A_n}$ is again an involution. 
\end{proof}

\begin{lem}
    Any involution in $\cfg T$ is in the closure of the derived subgroup of $\cfg T$.
\end{lem}

\begin{proof}
    Let $S\in \cfg T$ be an involution, and let $B\subset X$ be such that $\supp S=B\sqcup S(B)$. By Lemma \ref{lem: invo de support donné} there is an involution $U\in[T]$ whose support is $B$. By Lemma \ref{lem: invos approximées par cocycle borné}, we can find an increasing family $B_n$ of $U$-invariant subsets of $\supp U$ such that $B=\bigcup_nB_n$, and for all $n$, $U_{B_n}\in[T]_\infty\subset \cfg T$.
    
    Now let $C$ be such that $B=C\sqcup U(C)$ $C_n:=B_n\cap C$, and let $S_n=S_{(C_n\sqcup S(C_n))}$.
    Now define $U_n$ by
    $$U_n(x)=
    \begin{cases}
        U(x)\text{ if } x\in B_n,\\
        SUS(x)\text{ if } x\in S(B_n),\\
        x\text{ otherwise.}
    \end{cases}
    $$
    By construction, the commutator $[S_n, U_n]$ is the transformation induced by $S$ on the $S$-invariant set $A_n := B_n \sqcup S(B_n)$, and we have $\bigcup_n A_n=\supp S$, we may conclude that $[S_n,U_n]\to S$ by induction friendliness.
\end{proof}

\begin{lem}
    Every periodic element of $\cfg T$ is a product of two involutions of $\cfg T$.
\end{lem}

\begin{proof}
    We first start by proving it when $U$ is an $n$-cycle. In that case, let $A\subset X$ be such that 
    $$\supp U=\bigsqcup_{i=1}^n U^i(A).$$ 
    For $1\le i\le n-1$ we define $U_i\in\cfg T$ to be the involution defined by
    $$U_i(x)=
    \begin{cases}
        U(x)\text{ if }x\in U^i(A),\\
        U\inv(x)\text{ if }x\in U^{i+1}(A),\\
        x\text{ otherwise.}
    \end{cases}$$
    From the identity
    $$(1\quad 2\quad \cdots \quad n)=(1\quad 2)(2\quad 3)\cdots (n-1\quad n)$$
    in $\mathfrak S_n,$ we see that $U=U_1U_2\cdots U_{n-1}$, and also that if $i$ and $j$ have the same parity, we either have $i=j$ or $U_i$ and $U_j$ have disjoint supports. In any case, letting $V=U_2\cdots U_{2\lceil n/2\rceil-2}$ and $W=U_1U_3\cdots U_{2\lceil n/2\rceil-1}$, we do have $U=VW$, and both $V$ and $W$ are involutions in $\cfg T$

    If now $U\in \cfg T$ is any periodic map, let $A_n$ be the set of elements whose orbit under $U$ is of size exactly $n$. By periodicity, $\supp U=\bigsqcup_{n\ge 2}A_n$, and each of these sets is $U$-invariant. Hence by induction friendliness, 
    $$U=\lim_{n\to\infty} U_{\bigsqcup_{i=2}^nA_n}= \lim_{n\to\infty}\prod_{i=2}^n U_{A_i}.$$
    Applying our previous findings to the $U_{A_i}$, take $V_i,W_i$ to be involutions (supported on $A_i$) such that $U_{A_i}=V_iW_i$. In particular, as $A_i\cap A_j=\emptyset$, $V_i$ and $W_i$ commute with any $V_j,W_j$ for $j\ne i$. Hence for any $n$, we have 
    $$U_{A_1}\cdots U_{A_n}=V_1W_1\cdots V_nW_n=V_1\cdots V_nW_1\cdots W_n.$$
    As $\mu(\bigsqcup_{i\ge n}A_i)\tendv 0$, $(V_1\cdots V_n)_n$ converges to some involution $V\in [T]$. That $V$ is indeed in $\cfg T$ comes from the fact that for any $x\in X$, there is some $N\in\m N$ such that for any $n\ge N$, all the orbits of $U_{A_n}$ are confined to $T^{\m Z_{<0}}(x)$ or to $T^{\m Z_{>0}}(x)$ (and then it is also the case for the orbits of $V_i$).
\end{proof}

\begin{theo}\label{theo: plein de sgs égaux au dérivé}
    The following normal subgroups of $\cfg T$ are all equal:
    \begin{enumerate}
        \item The subgroup generated by the periodic elements,
        \item The subgroup generated by the involutions,
        \item The closure of the derived subgroup,
        \item The kernel of the index map.
    \end{enumerate}
\end{theo}

\begin{proof}
    Let $G_{per}$ be the subgroup generated by the periodic maps in $\cfg T$, and let $G_{inv}$ be the subgroup generated by the involutions in $\cfg T$ and $D(\cfg T)$ be the derived subgroup. 
    By the two previous lemmas, we have
    $$G_{per}\le G_{inv}\le \overline{D(\cfg T)}.$$
    Moreover, $I$ is a continuous mapping into an abelian group, so $\overline{D(\cfg T)}\le \ker I$.
    
    Note that $G_{per}$ is a normal subgroup of $\cfg T$, and that by \ref{pro: cfg eng par perio & T}, any $S\in \cfg T$ will be equal to $T^n$ modulo $G_{per}$ for some $n$. But since $G_{per}\le \ker I$ and $I(T)=1\ne 0$, if $S\in \ker I$ we must have $n=0$, i.e. $S\in G_{per}$. Hence $\ker I\le G_{per}$, and we are done.
\end{proof}

\subsection{Transitivity results of the boolean action of \texorpdfstring{$\cfg T$}{[T]com}}
We prove here that "almost" every set can be sent to any other by an element of the commensurating full group.

More precisely, 
\begin{pro}\label{pro: transi de cfg en pmp}
    Let $T\in\Aut(X,\mu)$ be ergodic. Let $A,B\subset X$ be such that $\mu(A)<\mu(B)$. Then there is $S\in\cfg T$ such that $S(A)\subset B$.
\end{pro}
The author thanks François Le Maître for suggesting the proof of this result. 
In order to be able to prove this result, we need the following lemma:
\begin{lem}
    Let $\mc R$ be a finite classes Borel equivalence relation on $X$. Assume $A,B\subset X$ are such that for a.e. $x\in X$,
    $$|[x]_\mc R\cap A|\le |[x]_\mc R\cap B|.$$
    Then there exists $\phi\in[\mc R]$ such that $\phi(A)\subset B$.
\end{lem}

\begin{proof}
    Choose an ordering of each $\mc R$-class. Since they are all finite, this can be done in a measurable way. Now define $\phi_1:A\to B$ which, for each $\mc R$-class, sends the first element of $A$ to the first element of $B$, the second element of $A$ to the second element of $B$, etc. The condition ensures that on each $\mc R$ class,we run out of elements of $A$ before we run out of elements of $B$, so that $\phi_1$ is a well-defined element of $[[\mc R]]$ the pseudo-full group of $\mc R$.
    
     We may now apply the same construction to $X\setminus A$ and $X\setminus \phi_1(A)$ (which do satisfy the hypothesis of the lemma), to obtain $\phi_2\in[[\mc R]]$.
    
    It is then easily checked that $\phi=\phi_1\amalg \phi_2$ is as wanted.
\end{proof}

With this lemma in hand, we are now able to prove Proposition \ref{pro: transi de cfg en pmp}.

\begin{proof}
    By Birkhoff's ergodic theorem, for a.e. $x\in X$, there exists some integer $N=N(x)$ such that for all $n\ge N$, 
    $$\sum_{k=0}^{n-1}\m 1_A(T^k(x))<\sum_{k=0}^{n-1}\m 1_B(T^k(x)).$$
    Let $N$ be such that the set $X_N$ of $x\in X$ for which the above property is verified with $N(x)=N$ is of positive measure. Let $C\subset X_N$ be such that $C,T(C),\ldots T^{N-1}(C)$ are all disjoint. For the order relation $\le_T$ on $X$, the set $C$ separates the orbits of $T$ into finite segments $[c_1,c_2[$ by ergodicity. Now we may define the equivalence relation $\mc R_C$ by "being in the same $C$-segment". It is then easy to check that $\mc R_C$ satisfies the hypotheses of the previous lemma, which we readily apply to get $S\in[\mc R_C]$ such that $S(A)\subset B$.
    
    Moreover, $\mc R_C\subset \mc R_T$, so that $S\in[T]$, and the action of $S$ on a $T$-orbit is a product of permutations on disjoint finite segments: it is in $\szn$. As such, $S\in\cfg T$, and even in $\ker I$ since $S$ is periodic.
\end{proof}

Note that the exact proof carries over if we only assume that $\mu$ is $\sigma$-finite if $A,B$ have finite measure thanks to the ratio ergodic theorem.

\section{A bounded complete invariant of flip conjugacy}

We now set out to proving that $\dcfg T$ is both a bounded Polish group and an invariant of flip conjugacy. The fact that $\dcfg T$ is bounded is the only result in this paper which is not analogous to the case of the integrable full group, for which $\overline{D([T]_1)}$ is proven to be unbounded in \cite[Corollary 6.7]{LM21}.

\subsection{The derived commensurating full group is bounded}\label{ssec: dcfg bounded}

\begin{lem}\label{lem: découpe les périos en petits bouts}
    Let $U\in [T]$ be periodic. For any integer $n\ge 1$ we can find a $U$-invariant measurable partition $A_0,\ldots,A_n$ of $X$ such that $\mu(A_i)\le \frac 1n$ for all $i$.
\end{lem}

\begin{proof}
    For $i\in\m N$, let $B_i=\{x\in X,\ord_U(x)=i\}$ be the set of elements whose orbit under $U$ is of size exactly $i$. Since $U$ is periodic, the $B_i$ form a $U$-invariant partition of $X$. 
    
    Let $N\in\m N$ be such that 
    $$\mu\left(\bigsqcup_{i\ge N}B_i\right)\le \frac 1n,$$
    we let $A_0:=\bigsqcup_{i\ge N}B_i$.
    Now fix a fundamental domain $D$ for the action of $U$ on $Y:=X\setminus A_0$. It has positive measure, as $\mu(Y)\ge \frac {n-1}n$ and $U$ has finite order on $Y$.
    
    Identifying $(D,\mu)$ with $([0,\mu(D)],\mathrm{Leb})$, let
    $\phi:[0,\mu(D)]\to [0,\mu(Y)]$ be the map defined by $\phi(t):=\mu(U^\m N([0,t]))$. $\phi$ is increasing, and 
    $$\phi(t+\epsilon)-\phi(t)\le \sum_{i=0}^N \mu(U^i(]t,t+\epsilon]).$$
    All of the maps $t\mapsto \mu(U^i([0,t]))$ are continuous, so $\phi$ is continuous, hence bijective. Let $t_i:=\phi\inv (\frac in)$ for $0\le i\le n-1$. Let also $t_n=\mu(D)$.
    
    Now we may define the $A_i$ by letting
    $$A_i:=U^\m N([t_{i-1},t_i[).$$
    It is straightforward to see that the $A_i$ are as desired.
\end{proof}

We now need the commensurated analog of \cite[Proposition 6.3]{LM21}, the proof being almost identical.\medskip

\begin{pro}\label{pro:perios denses dans le dérivé}
    The periodic transformations are dense in the topological derived subgroup $\overline{D(\cfg T)}$.
\end{pro}

\begin{proof}
    $\mc R_T$ being hyperfinite, let $\mc R_n$ be an increasing sequence of finite equivalence relations on $X$ such that $\mc R_T=\bigcup_n \mc R_n$.
    
    We first prove that $\bigcup_n [\mc R_n]\cap \overline{D(\cfg T)}$ is dense in $\overline{D(\cfg T)}$. Since $\overline{D(\cfg T)}$ is topologically generated by involutions, we need only approximate those by elements of $\bigcup_n [\mc R_n]\cap D(\cfg T)$ to be done. Let $U\in \overline{D(\cfg T)}$ be an involution, and let 
    $$A_n:=\{x\in X,\:x\mc R_n Ux\}.$$
    The $A_n$ are an increasing union of $U$-invariant sets, and $\bigcup_n A_n=X$ since $\bigcup \mc R_n=\mc R_T$. By induction friendliness, $U_{A_n}\in \cfg T$ and $U_{A_n}\to U$, and $U_{A_n}\in [\mc R_n]$ by definition. Now $U_{A_n}$ being an involution, $U_{A_n}\in \ker I=\overline{D(\cfg T)}$ and we are done.\medskip

    Now all elements of $[\mc R_n]$ are periodic, and the result follows.
\end{proof}

\begin{theo}\label{theo: dcfg T bounded}
    $\overline{D(\cfg T)}$ has bounded geometry.
\end{theo}

\begin{proof}
    Let $V$ be a neighborhood of $\id$ in $\overline{D(\cfg T)}$, we will prove that there is some $n$ for which $\overline{D(\cfg T)}=V^n$. Denote by $P$ the set of periodic maps in $\overline{D(\cfg T)}$. First, note that if we prove that there is some $n$ such that $P\subset V^n$, we would be through. This comes from Proposition \ref{pro:perios denses dans le dérivé}, and the fact that multiplying a dense subset by a neighborhood of $\id$ always gives the whole group, so $V^{n+1}\supset VP=D(\cfg T)$.
    
    Fix a compatible distance $d$ on $\szn$. The family 
    $$V_\epsilon =\{S\in \cfg T, \mu(\{d(\id,\Phi(S)(x))>\epsilon\})<\epsilon\}$$
    is a basis of neighborhoods of $\id$, let $n$ be such that $V_{1/n}\subset V$. Let $U\in P$. By Lemma \ref{lem: découpe les périos en petits bouts}, we can find a $U$-invariant partition $A_i,\:0\le i\le n$ with $\mu(A_i)\le \frac 1n$. Write 
    $$U=\prod_i U_{A_i}.$$
    Since $\mathrm{supp}(U_{A_i})\subset A_i$, $U_{A_i}\in V_{1/n}\subset V$. Hence $U\in V^{n+1}$ and we are done.
\end{proof}

\begin{cor}
    $\cfg T$ is quasi-isometric to $\m Z$.
\end{cor}

\begin{proof}
    Let $|\cdot|$ be any bounded compatible norm on $\cfg T$. We let
    $$\|S\|=|S|+|I(S)|,$$
    and we will prove that $\|\cdot\|$ is a maximal norm on $\cfg T$. In particular, 
    $$I:\cfg T\to\m Z$$
    will be a quasi-isometry.
    
    Let $\|\cdot\|'$ be a compatible norm on $\cfg T$. $\|\cdot\|'$ is bounded in restriction to $\dcfg T$, say by $C$.
    
    By Theorem \ref{theo: plein de sgs égaux au dérivé}, any $S\in \cfg T$ can be written as $S=UT^n$, where $U\in\dcfg T$ and $n=I(S)$. Hence,
    $$\|S\|'\le \|U\|'+|n|\underbrace{\|T\|'}_{C'}\le C+|I(S)|C'\le C+C'\|S\|$$
    and we are done.
\end{proof}

\subsection{Complete invariants of flip conjugacy}\label{ssec: complete invts of flip conj}
\begin{pro}\label{pro: Z-orbites -> flip conj}
    Let $U\in \cfg T$ be such that $T$ and $U$ have the same $\m Z$-orbits. Then $T$ and $U$ are flip-conjugate by an element of $[T]$.
\end{pro}

\begin{proof}
    Write $U=U_pU_+U_-$ according to Proposition \ref{pro: decomp perio/+/-}. Since $U$ has the same orbits as $T$, none of them are finite, hence $U_p=\id$. Moreover, $\supp U_+$ and $\supp U_-$ are disjoint and $U$-invariant, hence $T$-invariant: by ergodicity one of them has to be null. Thus, up to changing $U$ by $U\inv$, we may assume that $U$ is almost positive. \medskip

    We have thus arrived at the assumption $|T^\m N(x)\triangle U^\m N(x)|<\infty$. This is the hypothesis of \cite[Theorem A.1]{CJMT23}, and the proof does not use anywhere the fact that $T$ is pmp. We can thus conclude that $T$ and $U$ are conjugated by an element $S$ of $[T]$.
\end{proof}

We will need the following definition, due to Fremlin:

\begin{dfn}
    Let $G\le \Aut(X,[\mu])$ be a subgroup. $G$ is said to have \emph{many involutions} if for any non-null Borel subset $A$ of $X$, there is a non-trivial involution $U$ with $\supp U\subset A$.
\end{dfn}

\begin{ex}
    $[T]$ and $\cfg T$ have many involutions. This comes from Lemma \ref{lem: invo de support donné} for $[T]$, and we additionally need Lemma \ref{lem: invos approximées par cocycle borné} for $\cfg T$.
\end{ex}

The following is \cite[384D]{Fre04}:
\begin{theo}\label{theo:Fremlin}
    Let $G, H \le \Aut(X, [\mu])$ be two groups with many involutions. Then any isomorphism between $G$ and $H$ is the conjugacy by some non-singular transformation: for any group isomorphism $\rho : G \to H$, there is $S \in \Aut(X,[\mu])$ such that for all $T \in G$,
    $\rho(T) = STS\inv$.
\end{theo}

This is the commensurated equivalent of \cite[4.2]{LM18}.

\begin{theo}\label{theo: cfg T invar flip conj}
    Let $T_1,T_2\in\Aut(X,[\mu])$ be ergodic transformations. The following are equivalent:
    \begin{enumerate}
        \item $\cfg {T_1}$ and $\cfg{T_2}$ are isomorphic as abstract groups,
        \item $\cfg {T_1}$ and $\cfg{T_2}$ are isomorphic as topological groups,
        \item $T_1$ and $T_2$ are flip-conjugate.
    \end{enumerate}
\end{theo}

\begin{proof}
    We have 3 implies 2 and 2 implies 1. Now assume 1. By Theorem \ref{theo:Fremlin}, there is $S\in\Aut(X,[\mu])$ such that $\cfg {T_2}=S\cfg {T_1}S\inv$. In particular, $ST_1S\inv\in\cfg {T_2}$.
    
    Assume $ST_1S\inv$ does not have the same orbits as $T_2$. Then $S\inv T_2S$ would not have its orbits included in those of $T_1$, contradiction. Hence $ST_1S\inv$ does indeed have the same orbits as $T_2$. We can thus apply Proposition \ref{pro: Z-orbites -> flip conj} to conclude that $ST_1S\inv$ and $T_2$ are flip conjugate, and we are done.
\end{proof}

We now set out to proving the same result for $\dcfg T$ instead of $\cfg T$. For this we will need the following (well-known) technical lemma:

\begin{lem}\label{lem: partition foireuse}
    Let $T\in\Aut(X,[\mu])$. Then there is a Borel partition $(A_1,A_2,B_1,B_2,B_3)$ of $\supp T$ such that $T(A_1)=A_2$, $T(B_1)=B_2$ and $T(B_2)=B_3$.
\end{lem}

\begin{proof}
    See for instance \cite[2.7]{LM18}, along with the Remark below it.
\end{proof}

\begin{lem}\label{lem: fin full closure of derived sg}
    The finitely full closure $\mathrm{cl}(\dcfg T)$ of $\dcfg T$ is $\cfg T$.
\end{lem}

\begin{proof}
    $\cfg T$ is finitely full, and generated by $\dcfg T$ and $T$, so we only need to prove that $T$ can be obtained as a finite cutting and pasting of elements of $\dcfg T$.
    
    To do so we proceed in the same way as \cite[6.13]{LM21}. Apply Lemma \ref{lem: partition foireuse} to $T$ to get a partition $(A_1,A_2,B_1,B_2,B_3)$ of $\supp T$ such that $T(A_1)=A_2$, $T(B_1)=B_2$ and $T(B_2)=B_3$. Then define $U_1,U_2,U_3$ by
    
    \begin{align*}
    U_1(x)&=\begin{cases}
        T(x)\text{ if }x\in A_1\sqcup B_1,\\
        T\inv(x)\text{ if }x\in A_2\sqcup B_2,\\
        x\text{ otherwise,}
    \end{cases}\\
    U_2(x)&=\begin{cases}
        T(x)\text{ if }x\in B_2,\\
        T\inv(x)\text{ if }x\in B_3,\\
        x\text{ otherwise,}
    \end{cases}\\
    U_3(x)&=\begin{cases}
        T(x)\text{ if }x\in A_2\sqcup B_3,\\
        T\inv(x)\text{ if }x\in T(A_2)\sqcup T(B_3),\\
        x\text{ otherwise.}
    \end{cases}
    \end{align*}
    The $U_i$ are involutions which are all contained in $\dcfg T$, and
    $$T=U_1|_{A_1\sqcup B_1}\amalg U_2|_{B_2}\amalg U_3|_{A_2\sqcup B_3}\amalg \id_{(X\setminus \supp T)}$$
    so we are done.
\end{proof}

We are now ready to prove the same result as Theorem \ref{theo: cfg T invar flip conj} with $\dcfg T$. 

\begin{theo}\label{theo: dcfg T invar of flip conj}
    Let $T_1,T_2\in\Aut(X,\mu)$ be ergodic transformations. The following are equivalent:
    \begin{enumerate}
        \item $\dcfg {T_1}$ and $\dcfg{T_2}$ are isomorphic as abstract groups,
        \item $\dcfg {T_1}$ and $\dcfg{T_2}$ are isomorphic as topological groups,
        \item $T_1$ and $T_2$ are flip-conjugate.
    \end{enumerate}
\end{theo}

\begin{proof}
    Again, we have 3 implies 2 and 2 implies 1. Now assume 1. Our goal will be to prove that the whole commensurated subgroups are still isomorphic, and then to conclude via Theorem \ref{theo: cfg T invar flip conj}.
    
    In order to do that, note that since every involution of $\cfg T$ are in $\dcfg T(=\ker I)$, so that $\dcfg T$ again has many involutions. We can thus apply Theorem \ref{theo:Fremlin} again, to get $S\in \Aut(X,[\mu])$ such that $\dcfg {T_2}=S\dcfg{T_1}S\inv$.
    
    Now conjugation by an element of $\Aut(X,[\mu])$ preserves finitely full closures, so by Lemma \ref{lem: fin full closure of derived sg}, 
    $$S\cfg {T_1}S\inv=S\mathrm{cl}(\dcfg {T_1})S\inv=\mathrm{cl}(S\dcfg {T_1}S\inv)=\mathrm{cl}(\dcfg{T_2})=\cfg {T_2}.$$
    In particular, $\cfg {T_1}$ and $\cfg {T_2}$ are isomorphic, and we conclude by Theorem \ref{theo: cfg T invar flip conj}.
\end{proof}

\section{Topological simplicity}\label{sec: Topo simp}
We now prove that the derived commensurating full group of $T$ is topologically simple, that is, it has no non trivial closed normal subgroups.\medskip

The following is the equivalent of \cite[3.21]{LM18} in the measure-class preserving case.

\begin{lem}\label{lem: mm mesure+disj-> approx conj}
    Let $T\in\Aut(X,[\mu])$ be ergodic, and let $\nu\sim\mu$ be a $T$-invariant measure. Let $U,V\in\cfg T$ be two involutions with disjoint supports such that $\nu(\supp U)=\nu(\supp V)$. Then $U$ is approximately conjugate to $V$, i.e. there is a sequence $S_n\in \dcfg T$ such that $S_nUS_n\inv \to V$.
    
    If $T$ does not admit any such invariant measures, any $U,V\in \cfg T$ with disjoint (nonempty) supports are approximately conjugates.
\end{lem}

\begin{proof}
    We begin by choosing $A,B$ such that $\supp U=A\sqcup U(A)$ and $\supp V=B\sqcup V(B)$. In either case, $A$ and $B$ still satisfy the hypotheses of Lemma \ref{lem: invo qui échange A et B}, so that there is an involution $S\in[T]$ such that $S(A)=B$ and $\supp S=A\sqcup B$. We now proceed in much the same way as in \cite[Lemma 3.21]{LM18} to conclude.
    
    For $n\in\m N$, let $A_n=\{x\in A,|c_S(x)|<n\}$ and $B_n=\{x\in B,|c_S(x)|<n\}$, so that $S(A_n)=B_n$, $S_{A_n\cup B_n}\in [T]_\infty$ and $\bigcup_n A_n=A,\:\bigcup_n B_n=B$. Now define a sequence of involutions $S_n$ by
    $$S_n(x)=\left\{
    \begin{array}{cl}
        S(x) &\text{ if } x\in A_n\sqcup B_n,\\
        VSU(x) &\text{ if } x\in U(A_n),\\
        USV(x) &\text{ if } x\in V(B_n),\\
        x &\text{ otherwise.}
    \end{array}\right.$$
    Note that $S_n\in \cfg T$ by finite fullness. Moreover, a straightforward computation yields that
    $$S_nUS_n(x)=\left\{
    \begin{array}{cl}
        U(x) &\text{ if } x\in A\cup U(A)\setminus (A_n\cup U(A_n)),\\
        V(x) &\text{ if } x\in B_n\cup V(B_n),\\
        x &\text{ otherwise.}
    \end{array}\right.$$
    Hence $S_nUS_n=V_{B_n\cup V(B_n)}U_{(A\cup U(A))\setminus (A_n\cup U(A_n))}$. Now by induction friendliness, $V_{B_n\cup V(B_n)}\to V$ and $U_{(A\cup U(A))\setminus (A_n\cup U(A_n))}\to \id,$ so that $S_nUS_n\to V$ as hoped.
\end{proof}

\begin{lem}\label{lem: <<invo>>=tout le monde, type III}
    Let $T\in\Aut(X,[\mu])$ be ergodic with no invariant measures, and let $U\in\cfg T$ be an involution such that $\supp U$ is neither null nor conull. Then the closure of the subgroup generated by conjugates of U by elements of $\dcfg T$ is $\dcfg T$ itself.
\end{lem}

\begin{proof}
    Let $G$ be the closure of the subgroup generated by conjugates of $U$ by elements of $\dcfg T$, and write $S\sim S'$ to say that $S$ and $S'$ are approximately conjugated. Notice that if we have a chain $U\sim S_1\sim\cdots \sim S_k\sim V$, then $V\in G$.
    
    Since $\dcfg T$ is topologically generated by involutions, we only need to prove that any involution $V\in\cfg T$ is in $G$.
    
    First assume that $\supp V\ne X$, and let $A:=X\setminus \supp U$. If $B:=A\setminus \supp V$ is non-null, then let $W\in\cfg T$ be any involution supported on $B$ ($\cfg T$ has many involutions). Then $U\sim W\sim V$ by Lemma \ref{lem: mm mesure+disj-> approx conj}.
    
    Otherwise $A\subset \supp V$. Then let $S$ be any involution supported on $A$, and $S'$ be any involution supported on $X\setminus \supp V$. Then $U\sim S\sim S'\sim V$ and we are done.
    
    Now for the case when $\supp V$ is conull, let $A$ be a fundamental domain for $V$, and split $A=B\sqcup C$ with $B,C$ non-null. Then the first part of the proof shows that $U\sim V_{B\sqcup V(B)}$ and $U\sim V_{C\sqcup V(C)}$. As $V=V_{B\sqcup V(B)}V_{C\sqcup V(C)}$, we are done.
\end{proof}

\begin{lem}\label{lem: <<invo>>=tout le monde, type II}
    Let $T\in\Aut(X,[\mu])$ be ergodic, and let $\nu\sim\mu$ be a $T$-invariant measure. Let $U\in\cfg T$ be an involution with $\nu(X\setminus \supp U)\ge \nu(\supp U)$. Then the closure of the subgroup generated by conjugates of U by elements of $\dcfg T$ is $\dcfg T$ itself.
\end{lem}

\begin{proof}
    Let $B=\supp U$, and $A$ such that $B=A\sqcup U(A)$ and let $V\in [T]$ be an involution such that $\supp V=B\sqcup V(B)$, which exists by Lemma \ref{lem: invo qui échange A et B} and the assumption on $\supp U$.
    
    For $n\in \m N$, let $B_n$ be defined by
    $$B_n=\{x\in B,\max(|c_V(x)|,|c_V(U(x))|)<n\}.$$
    Since $\bigcup_nB_n=B$, there is $N$ such that $B':=B_N$ is not null. Note that $B'$ is $U$-invariant, so that $U':=U_{B'}$ is still an involution, and $A':=A\cap B'$ is still a fundamental domain for $U'$ restricted to its support.\medskip

    Let $G$ be the closure of the subgroup generated by conjugates of $U$ by elements of $\dcfg T$. 
    Also, let $\Gamma=\langle U,V\rangle$. Note that $VUV$ is an involution supported on $V(B)$, which is disjoint of $\supp U$, so that $UVUV=VUVU$ is again an involution. In particular, $(UV)^4=\id$, and $\Gamma\simeq D_8\simeq \m Z/4\m Z\rtimes\m Z/2\m Z$. In particular, $|\Gamma|=8$.
    
    Now since $\dcfg T$ is topologically generated by involutions, we only need to prove that any involution is in $G$. Thus let $W\in\cfg T$ be any involution, and $D$ such that $\supp W=D\sqcup W(D)$.
    
    We can decompose 
    $$D=\bigsqcup_{n\in\m N} D_n$$
    with $\nu(D_n)\le \min(\nu(A'),1)/20$ for each $n$.
    Letting $W_n:=W_{D_n\sqcup W(D_n)}$, we have 
    $$W=\lim_{n\to\infty}\prod_{k=0}^nW_n$$ by induction friendliness.
    Hence we only need to prove that $W_n\in G$ for all $n$. To this end, fix any $n\in\m N$, and let $E\subset A'$ be such that $\nu(E)=\nu(D_n)/2$ and $E$ is disjoint from $\Gamma(D_n\sqcup W(D_n))$. Such an $E$ exists because $\nu(\Gamma(D_n\sqcup W(D_n)))\le 16\nu(D_n)<\infty$, so that $\nu(A'\setminus\Gamma(D_n\sqcup W(D_n)))\ge 4\nu(D_n)\ge \nu(D_n)/2$. We thus also get that $\Gamma(E)$ is disjoint from $D_n\sqcup W(D_n)=\supp W_n$.
    
    Furthermore, note that since $E\subset B$, $\tilde E:=E\sqcup U(E)\sqcup V(E\sqcup U(E))$ is $\Gamma$-invariant. Let $\tilde V$ be the involution $\tilde V=V_{\tilde E}$, and note that $\tilde V\in \cfg T$ by the fact that $E\sqcup U(E)\subset B'$. An easy computation then yields $U\tilde VU\tilde V=U_{\tilde E}$, so that $U_{\tilde E}\in G$. Moreover, 
    $$\nu(\supp U_{\tilde E})=\nu(\tilde E)=4\nu(E)=2\nu(D_n)=\nu(D_n\sqcup W(D_n))=\nu(\supp W_n).$$
    Note that by construction, $\tilde E$ and $\supp W_n$ are disjoint, so we may infer by Lemma \ref{lem: mm mesure+disj-> approx conj} that $W_n$ is approximately conjugate to $U_{\tilde E}$. Since $U_{\tilde E}\in G$, we get that $W_n\in G$, which concludes the proof.
\end{proof}

\begin{theo}\label{theo: dcfg is top simple}
    For any ergodic $T\in\Aut(X,[\mu])$, $\dcfg T$ is topologically simple.
\end{theo}

\begin{proof}
    Let $S\in \dcfg T$ be non-trivial. We will prove that there exists an involution $V\in\ll \!S\!\gg$ which satisfies the hypotheses of either Lemma \ref{lem: <<invo>>=tout le monde, type II} or \ref{lem: <<invo>>=tout le monde, type III}, depending on the type of $T$.
    
    Let $\nu\sim\mu$ be a $T$-invariant measure. Up to rescaling, we may assume $\nu(X)\in\{1,\infty\}$. We then find some $A\subset X$ such that $A\cap S(A)=\emptyset$ and $0<\nu(A)\le 1/4$, which is always possible since $S\ne\id$.
    
    Now let $U\in\cfg T$ be an involution supported on $A$, so that $SUS\inv$ and $U$ are involutions in $\cfg T$ with disjoint support, so that $V=USUS\inv$ is an involution with $\supp V\subset A\sqcup S(A)$. Moreover, $USU$ and $S\inv$ are both members of $\ll \!S\!\gg$, so that $V\in \ll \!S\!\gg$. We now apply Lemma \ref{lem: <<invo>>=tout le monde, type II} to conclude.\medskip

    If $T$ has no invariant measures in the class of $\mu$, the proof is basically the same. Here we only need to find $A$ such that $A\cap S(A)=\emptyset$ and $A\cup S(A)$ isn't conull to be able to apply Lemma \ref{lem: <<invo>>=tout le monde, type III}.
\end{proof}

\bibliographystyle{alpha}  
\bibliography{biblio}

\end{document}